\documentclass[11pt]{article}

\usepackage{amsmath}    
\usepackage{graphicx}   
\usepackage{verbatim}   
\usepackage{color}      
\usepackage{hyperref}   

\usepackage{amssymb}
\usepackage{amsthm}

\theoremstyle{theorem}
\newtheorem{Def}{Definition}[section]

\newtheorem{Prop}[Def]{Proposition}
\newtheorem{Lem}[Def]{Lemma}
\newtheorem{Thm}[Def]{Theorem}
\newtheorem{Cor}[Def]{Corollary}
\newtheorem{Ass}[Def]{Assumption}
\theoremstyle{definition}
\newtheorem{Rem}[Def]{Remark}

\newcommand{\p}{\mathbb{P}}
\newcommand{\e}{\mathbb{E}}
\newcommand{\real}{\mathbb{R}}
\newcommand{\n}{\mathbb{N}}

\newcommand{\1}{{\bf 1}}

\newcommand{\ksm}{K_\sigma}

\begin{document}

\title{On the Euler-Maruyama approximation for one-dimensional stochastic differential equations with irregular coefficients}
\author
{
	Hoang-Long Ngo\footnote{Hanoi National University of Education, 136 Xuan 
	Thuy - Cau Giay - Hanoi - Vietnam, email: $\qquad$ngolong@hnue.edu.vn}
	$\quad $ and $\quad$ 
	Dai Taguchi\footnote{Ritsumeikan University, 1-1-1 Nojihigashi, Kusatsu, 
	Shiga, 525-8577, Japan, email: dai.taguchi.dai@gmail.com }
}

\date{}
\maketitle
\begin{abstract}
We study the strong rates of the Euler-Maruyama approximation for one dimensional stochastic differential equations whose drift coefficient  may be neither continuous nor one-sided Lipschitz and diffusion coefficient is H\"older continuous. Especially, we show that the strong rate of the Euler-Maruyama approximation is 1/2 for a large class of  equations whose drift is not continuous. We also provide the strong rate for equations  whose drift is H\"older continuous and diffusion is nonconstant.
\end{abstract}

\section{Introduction} \label{Sec_1}
Let us consider the one-dimensional stochastic differential equation (SDE)
\begin{align}\label{SDE_1}
X_t=
x_0
+\int_0^t b(X_s)ds
+\int_0^t \sigma(X_s)dW_s,
~x_0 \in \real,
~t \in [0,T],
\end{align}
where $W:=(W_t)_{0\leq t \leq T}$ is a standard one-dimensional Brownian motion on a probability space $(\Omega, \mathcal{F},\p)$ with a filtration $(\mathcal{F}_t)_{0\leq t \leq T}$ satisfying the usual conditions. 

It is well-known that the solution to the SDE \eqref{SDE_1} is related to the Kolmogorov equation.
Stroock and Varadhan \cite{SV69} prove that if the drift coefficient $b$ is bounded, measurable and the diffusion coefficient $\sigma$ is bounded, uniformly elliptic and continuous, then a solution to the Kolmogorov equation $\frac{\partial u}{\partial t}+b\frac{\partial u}{\partial x}+\sigma^2 \frac{\partial^2 u}{\partial^2 x}=0$ with the boundary condition $u(T,x)=f(x)$, in the class $W^{1,2}_p$ with $p >3/2$ admits the stochastic representation (see also Theorem 1 in \cite{Zv}).
Zvonkin \cite{Zv} studied the existence and uniqueness of solution of the SDE \eqref{SDE_1} under very weak regularity assumption of coefficients $b$ and $\sigma$.
In particular, he showed that if $b$ is bounded, measurable and $\sigma$ is bounded, uniformly elliptic and  $(\alpha + \frac 12)$-H\"older continuous for some $\alpha \in [0, \frac 12]$ then equation \eqref{SDE_1} has a unique strong solution (see also Veretennikov \cite{V}).

Since the solution of (\ref{SDE_1}) is rarely analytically tractable, one often approximates $X=(X_t)_{0 \leq t \leq T}$ by using the Euler-Maruyama scheme given by 
\begin{align*}
X_t^{(n)} 
&= x_0 +\int_0^tb\left(X_{\eta _n(s)}^{(n)}\right)ds +\int_0^t \sigma\left(X_{\eta _n(s)}^{(n)}\right) dW_s,~t \in [0,T],
\end{align*}
where $\eta _n(s) = kT/n=:t_k^{(n)}$ if $ s \in \left[kT/n, (k+1)T/n \right)$. 

Both the strong and weak rates of convergence of $X^{(n)}$ to $X$ are known when $b$ and $\sigma$ satisfy some Lipschitz continuous condition (see \cite{KP,BT}).
It has been shown recently that 
there exist SDEs with smooth and bounded coefficients such that neither the EM approximation nor any approximation method based on finitely many observations of the driving Brownian motion can converge in absolute mean to the solution faster than any given speed of convergence (see \cite{HHJ, JMY}).
However, there are few results when $b$ is irregular. 
When $b$ is not continuous, most of the works so far need the assumption that $b$ is one-sided Lipschitz to establish the rate of convergence (see \cite{G98, NT}).
Note that this one-sided Lipschitz condition  also plays an indispensable role to establish the rate of convergence for SDEs with super-linear growth coefficients (see \cite{HJK}, \cite{CJM}).
Outside the framework of one-sided Lipschitz, Halidias and Kloeden \cite{HK} showed that $X^{(n)}$ converges to $X$ in $L^2$-norm if $b$ is increasing, continuous from bellow  and $\sigma$ is Lipschitz continuous.
Since their proof uses upper and lower solutions of the SDEs and the Euler-Maruyama approximation, it is hardly possible to get any rate of convergence by using their method.  
Recently, Leobacher and Sz\"olgyenyi \cite{LeSz} studied the SDE \eqref{SDE_1} under the assumption that $b$ is piecewise Lipschitz, has a finite number of discontinuous points and $\sigma$ is Lipschitz and uniformly elliptic.
They introduced a clever way to transfer equation \eqref{SDE_1} to an equivalent  equation whose coefficients are Lipschitz continuous and therefore the new equation can be approximated by an Euler-Maruyama scheme with the standard rate of convergence $1/2$. 

The strong rates of the Euler-Maruyama approximation for SDEs with H\"older continuous diffusion coefficient were first established in \cite{Y, BBD, Gyongy}.
The main idea in \cite{Gyongy} is to  use the so-called Yamada-Watanabe approximation method to estimate the error. This remarkable idea has been developed in \cite{BY,NT, LT} to obtain strong rate under various assumptions on coefficients $b$ and $\sigma$.
It is undoubted  that Yamada-Watanabe approximation is still a key tool to deal with the H\"older continuity of $\sigma$ in this paper.

When $b$ is only H\"older continuous of order $\beta \in (0,1]$ and $\sigma$ is a non-zero constant, Menoukeu-Pamen and Taguchi \cite{MeTa} have used a PDE technique to show very recently that the strong rate of the Euler-Maruyama approximation is of order $\beta/2$.

In this paper, we will study the rates of strong convergence of the Euler-Maruyama approximation for  SDE \eqref{SDE_1} when the coefficients $b$ and $\sigma$ may have a very low regularity.
In particular, we consider the case that $\sigma$ is $(\alpha + \frac 12)$-H\"older continuous and $b = b_{A} + b_H$ where $b_{A}$ is, roughly speaking, a function of bounded variation on compact sets and $b_H$ is  H\"older continuous of some order $\beta \in (0,1]$.
Note that $b$ is not necessary continuous or one-sided Lipschitz function.
By introducing a new approach based on the removal drift transformation, we are able to establish the rates of convergence of $X^{(n)}$ to $X$ in $L^1$, $L^1$-sup and $L^p$-sup norm $(p\geq 2)$.
Our finding partly improves upon recent results in \cite{MeTa, NT, Gyongy} as well as the well-known ones in \cite{HK, G98} in the one-dimensional setting (see Remark \ref{compare}). It worth noting that SDEs with discontinuous drift appear in many applications such as mathematical finance, optimal control and interacting infinite particle systems \cite{AI, BSW, CE, CS, KR}.

The remainder of the paper is structured as follows.
In the next section we introduce some notations and assumption for our framework together with the main results.
All proof are deferred to Section 3.

\section{Main results}\label{sec:preliminaries}



\subsection{Notations}
For bounded measurable function $f$ on $\real$, we define $\|f\|_{\infty}:=\sup_{x \in \real} |f(x)|$.
We denote by $L^1(\real)$ the space of all integrable functions on $\real$ with semi-norm $\|f\|_{L^1(\real)}:=\int_{\real} |f(x)|dx$.
For $\beta \in (0,1]$, we denote by $H^{\beta}$ the set of all functions from $\real $ to $\real$ which are bounded and $\beta$-H\"older continuous, i.e., a function $f \in H^{\beta}$ iff
\begin{align*}
\|f\|_\beta := \|f\|_\infty+ \sup_{x,y \in \real, x \neq y} \frac{|f(x)-f(y)|}{|x-y|^{\beta}} < \infty.
\end{align*}

We recall the following  class of functions $\mathcal{A}$ which is first introduced in \cite{KMN} (see also \cite{NT}).
Let $\mathcal{A}$ be a class of all bounded measurable functions $\zeta : \real \to \real$ such that there exist a finite positive constant $K_{\mathcal{A}}$ and a sequence of functions $(\zeta _N)_{N \in \n} \subset C^1(\real)$ satisfying:
$$\begin{cases}
\mathcal{A}(i): & \zeta _N \to \zeta  \text{ in } L^1_{loc}(\real), $ as $ N \to \infty,\\
\mathcal{A}(ii): &\sup_{N \in \n} |\zeta _N(x)| +|\zeta (x)|\leq K_{\mathcal{A}},\\ 
\mathcal{A}(iii): & \displaystyle \sup_{N \in \n, a \in \real} \int_{\real}  |\zeta_N'(x+a)| e^{-|x|^2/u}dx \leq  (1+\sqrt{u})K_{\mathcal{A}} \quad \text{for all } u >0. 
\end{cases}$$
We denote $\|\zeta\|_\mathcal{A}$ the smallest constant $K_{\mathcal{A}}$ satisfying the above conditions. 
The class $\mathcal{A}$ will be used to model a part of the drift coefficient $b$.

It is easy to verify that the class $\mathcal{A}$ contains all $C^1(\real)$ functions which has the first order derivative of polynomially bounded.
Furthermore, the class $\mathcal{A}$ contains also some non-smooth  functions of the type $\zeta(x)=(x-a)^+ \wedge 1$ or $\zeta(x)=\1_{b<x<c}$ for some $a\in \real$, $b,c \in [-\infty,\infty]$.

The following propositions shows that this class  is quite large. 

\begin{Prop}\label{exp_1}
	(i) If $\xi, \zeta \in \mathcal{A}$ and $\alpha, \beta \in \real$, then $\xi \zeta \in \mathcal{A}$ and $\alpha \xi + \beta \zeta \in \mathcal{A}$.
	(ii) If $\zeta : \real \rightarrow \real$ is a bounded measurable and monotone function, then $\zeta \in \mathcal{A}$.
\end{Prop}

The proofs of Proposition \ref{exp_1} and further properties of class $\mathcal{A}$ were presented in \cite{KMN} and \cite{NT}.
It is worth noting that if a function $\zeta \in \mathcal{A}$ has a compact support then it follows from Theorem 3.9 \cite{AFP} that $\zeta$ is of bounded variation.
Therefore  class $\mathcal{A}$ does not contain class of H\"older continuous function $H^\beta$ for any $\beta \in (0,1)$.

\subsection{Main results}

We need the following assumptions on the coefficients $b$ and $\sigma$.

\begin{Ass} \label{Ass_1}
	We assume that the coefficients $b$ and $\sigma $ are measurable functions and satisfy the following conditions:
	\begin{itemize}
		\item[(i)] $b=b_A+b_H \in L^1(\real) $ where $b_A \in \mathcal{A}$ and $b_H \in H^{\beta}$ with $\beta \in (0,1]$.
		
		\item[(ii)] $\sigma$ is uniform elliptic and globally bounded and globally H\"older continuous: there exist real numbers $\ksm > 1$ and $\alpha \in [0,\frac12]$ such that  
		\begin{align*}
		\frac{1}{\ksm^2} \leq \sigma^2(x) \leq \ksm^2 \text{ for any } x \in \real,
		\end{align*}
	and $$|\sigma(x) - \sigma(y)| \leq \ksm |x-y|^{\frac 12 + \alpha} \text{ for any } x , y \in \real.$$
	\end{itemize}
\end{Ass}


We obtain the following results on the rate of the Euler-Maruyama approximation in both $L^1$-norm and $L^1$-$\sup$ norm.


\begin{Thm} \label{Main_1}
	Let Assumption \ref{Ass_1} hold.
	Then there exists a constant $C^*$ which depends on $K_\sigma, \|b_A\|_{\mathcal{A}}, \|b_H\|_\beta, \|b\|_{L^1(\real)}, T, x_0, \alpha$ and $\beta$ such that
	\begin{align} \label{est_l1}
	\sup_{\tau \in \mathcal{T}}\e[|X_{\tau}-X_{\tau}^{(n)}|] 
	&\leq \left\{ \begin{array}{ll}
	\displaystyle \frac{C^*}{\log n} &\text{ if } \alpha = 0,  \\
	\displaystyle \frac{C^*}{n^{\frac{\beta}{2}\wedge \alpha }} &\text{ if } \alpha \in (0, 1/2],
	\end{array}\right.
	\end{align}
	where $\mathcal{T}$ is the set of all stopping times $\tau \leq T$. Moreover, for any $\gamma \in (0,1)$, it holds
	\begin{align}  \label{est_l1b}
	\e[\sup_{0 \leq s \leq T} |X_s-X_s^{(n)}|^\gamma] 
	&\leq \left\{ \begin{array}{ll}
	\displaystyle \dfrac{2-\gamma}{1-\gamma}\Big(\frac{C^*}{\log n}\Big)^\gamma &\text{ if } \alpha = 0,  \\
	\displaystyle \dfrac{2-\gamma}{1-\gamma} \Big(\frac{C^*}{n^{\frac{\beta}{2}\wedge \alpha }}\Big)^\gamma &\text{ if } \alpha \in (0, 1/2].
	\end{array}\right.
	\end{align}
	
\end{Thm}

The assumption that $b \in L^1(\real)$ is in fact quite restricted  since it excludes  some simple function such as $b(x) = \1_{x \geq 0}$. Fortunately, removing that assumption does not affect much on the strong rate of convergence as shown in the following theorem.
\begin{Thm} \label{Main_4}
	Suppose that the drift coefficient $b=b_A+b_H$, where $b_A \in \mathcal{A}$ and $b_H \in H^{\beta}$ with $\beta \in (0,1]$.
	Let Assumption \ref{Ass_1} (ii)  hold.
	Then there exists a constant $C^*$ which depends on $K_\sigma, \|b_A\|_{\mathcal{A}}, \|b_H\|_\beta, T, x_0, \alpha$ and $\beta$ 
	 such that
	\begin{align} \label{est_l4}
	\sup_{\tau \in \mathcal{T}}\e[|X_{\tau}-X_{\tau}^{(n)}|] 
	&\leq \left\{ \begin{array}{ll}
	\displaystyle \frac{C^*e^{C^*\sqrt{\log (\log n)}}}{\log n	} &\text{ if } \alpha = 0,  \\
	\displaystyle \frac{C^* e^{C^*\sqrt{\log n}}}{n^{\frac{\beta}{2}\wedge \alpha }} &\text{ if } \alpha \in (0, 1/2].
	\end{array}\right.
	\end{align}
	Moreover, for any $\gamma \in (0,1)$, it holds
	\begin{align} \label{est_l4b}
		\e[\sup_{0 \leq s \leq T} |X_s-X_s^{(n)}|^\gamma] 
	&\leq \left\{ \begin{array}{ll}
	\dfrac{2-\gamma}{1-\gamma}\Big(\displaystyle \frac{C^*e^{C^*\sqrt{\log (\log n)}}}{\log n}\Big)^\gamma  &\text{ if } \alpha = 0,  \\
	\dfrac{2-\gamma}{1-\gamma} \Big(\displaystyle \frac{C^* e^{C^*\sqrt{\log n}}}{n^{\frac{\beta}{2}\wedge \alpha }} \Big)^\gamma &\text{ if } \alpha \in (0, 1/2].
	\end{array}\right.
	\end{align}
	
\end{Thm}

\begin{Rem}
	Note that the function $x \mapsto e^{\sqrt{\log x}}$ increases faster than any polynomial function of $\log x$ but slower than any polynomial of $x$, i.e, for any $\varepsilon >0$ and $k>0$, 
	$$\lim_{x \to +\infty} \frac{e^{\sqrt{\log x}}}{x^\varepsilon} = \lim_{x \to +\infty} \frac{(\log x)^k}{e^{\sqrt{\log x}}} = 0.$$
\end{Rem}
The estimates \eqref{est_l1b} and \eqref{est_l4b} become worst when $\gamma \uparrow 1$. Fortunately, we have the following bounds for the $L^1$-$\sup$ norm. 
\begin{Thm} \label{Main_2}
	Let Assumption \ref{Ass_1} hold.
	Then there exists a constant $C$ which depends on $K_\sigma, \|b_A\|_{\mathcal{A}}, \|b_H\|_\beta, \|b\|_{L^1(\real)}, T, x_0, \alpha$ and $\beta$ such that
	\begin{align*}
	\e[\sup_{0 \leq s \leq T} |X_s-X_s^{(n)}|] 
	&\leq \left\{ \begin{array}{ll}
	\displaystyle \frac{C}{\sqrt{\log n}} &\text{ if } \alpha = 0,  \\
	\displaystyle \frac{C}{n^{\alpha ( \beta \wedge 2\alpha) }} &\text{ if } \alpha \in (0, 1/2].
	\end{array}\right.
	\end{align*}
\end{Thm}

\begin{Thm} \label{Main_4b}
Under the assumption of Theorem \ref{Main_4}, 
there exists a constant $C$ which depends on $K_\sigma, \|b_A\|_{\mathcal{A}}, \|b_H\|_\beta, T, x_0, \alpha$ and $\beta$ such that
	\begin{align*}
	\e[\sup_{0 \leq s \leq T} |X_s-X_s^{(n)}|] 
	&\leq \left\{ \begin{array}{ll}
	\displaystyle \frac{Ce^{C\sqrt{\log (\log n)}}}{\sqrt{\log n}} &\text{ if } \alpha = 0,  \\
	\displaystyle \frac{Ce^{C\sqrt{\log n}}}{n^{\alpha ( \beta \wedge 2\alpha) }} &\text{ if } \alpha \in (0, 1/2].
	\end{array}\right.
	\end{align*}
\end{Thm}


The following $L^p$-norm estimation is useful to construct a Multi-level Monte Carlo simulation for $X$ (see \cite{G}).

\begin{Thm} \label{Main_3}
	Let Assumption \ref{Ass_1} hold.
	For any $p \geq 2$, then there exists a constant $C$ which may depend on $K_\sigma, \|b_A\|_{\mathcal{A}}, \|b_H\|_\beta, \|b\|_{L^1(\real)}, T, x_0, \alpha, \beta$ and $p$ such that
	\begin{align*}
	\e[\sup_{0 \leq s \leq T} |X_s-X_s^{(n)}|^p] 
	&\leq \left\{ \begin{array}{ll}
	\displaystyle \frac{C}{\log n} &\text{ if } \alpha = 0,  \\
	\displaystyle \frac{C}{n^{ \frac{\beta}{2} \wedge \alpha }} &\text{ if } \alpha \in (0, 1/2), \\
	\displaystyle \frac{C}{n^{\frac12 \wedge \frac{p\beta}{2} }} &\text{ if } \alpha =1/2.
	\end{array}\right.
	\end{align*}
\end{Thm}

If we suppose that $\sigma$ is Lipschitz continuous and $b \in H^\beta$, i.e. $b_{\mathcal{A}} \equiv 0$, then Theorem \ref{Main_3} implies the following result which improves the one in \cite{MeTa} for SDEs with non-constant diffusion.
\begin{Cor}\label{Main_0}
	Assume that $b \in L^1(\real) \bigcap H^{\beta}$ for some $\beta \in (0,1]$ and  the diffusion coefficient $\sigma$ is Lipschitz continuous and  uniformly elliptic.
	Then for any $p \geq 1$, there exists positive constant $C$ which depends on $K_\sigma, \|b\|_\beta, \|b\|_{L^1(\real)}, T, x_0, \alpha, \beta$ and $p$ such that
	\begin{align*}
		\e[\sup_{0 \leq s \leq T} |X_s - X_s^{(n)}|^{p}]
		\leq \frac{C}{n^{p \beta/2}}.
	\end{align*}
\end{Cor}

\begin{Rem} \label{compare}
\begin{enumerate}
\item Gy\"ongy \cite{G98}  studied the rate of convergence in the almost sure sense of the Euler-Maruyama approximation for SDEs with irregular drift. He showed that the rate is $1/4$ when $\sigma$ is locally Lipschitz and $b$ satisfies an one-sided Lipschitz type condition.
\item In the case that $\beta = 1$, the results of Theorems \ref{Main_1}, \ref{Main_2} and \ref{Main_3} were proven in \cite{NT} under a further assumption that $b$ is one-sided Lipschitz.
In this paper, thanks to the method of removal drift we are able to get rid of this assumption.
Note that if $\sigma$ is Lipschitz function, the strong rate of the Euler-Maruyama approximation mentioned in Theorem \ref{Main_1} is $1/2$.
\end{enumerate}
\end{Rem}

\section{Proof of the main Theorems} \label{sec:proofs}

\subsection{Some auxiliary estimates}

The following lemma is a key estimation for proving the main theorems.
The proof is quite similar to the one of Lemma 3.5 in \cite{NT} and will be omitted.

\begin{Lem}\label{Lem}
	Assume that $b$ is bounded, measurable and $\sigma$ satisfies Assumption  \ref{Ass_1} (ii).
	Suppose that $\zeta \in \mathcal{A}$.
	Then for any $q \geq 1$, there exists $C \equiv C(T,\ksm, \|\zeta\|_{\mathcal{A}},\|b\|_{\infty},x_0,q)$ such that
	\begin{align}\label{bxsbxeta}
	\int_0^T\e[| \zeta(X_s^{(n)}) - \zeta(X_{\eta_n(s)}^{(n)})|^q]ds \leq \frac{C}{\sqrt{n}}.
	\end{align}
\end{Lem}
	
\begin{Rem}
	Note that the estimate \eqref{bxsbxeta} is tight (see, Remark 3.6 in \cite{NT}).
\end{Rem}
	
The following estimation is standard (see Remark 1.2 in \cite{Gyongy}).
\begin{Lem} \label{Lem_1}
	Suppose that $b$ and $\sigma$ are bounded, measurable.
	Then for any $q>0$, there exist $C \equiv C(q,\|b\|_{\infty}, \|\sigma\|_{\infty}, T) $ such that
	\begin{align*}
	\sup_{t \in [0,T]} \e[|X_t^{(n)}-X_{\eta_n(t)}^{(n)}|^q]\leq \frac{C}{n^{q/2}}.
	\end{align*}
\end{Lem}

\subsection{The method of removal of  drift}
The following removal of drift transformation plays a crucial role in our argument. 
Under the assumption that  $b \in L^1(\real)$  and $\sigma^2$ is uniformly elliptic, the following functions 
\begin{align*}
f(x) : = -2 \int_0^x \frac{b(y)}{\sigma^2(y)} dy \text{ and } \varphi (x) := \int_0^x \exp(f(y)) dy.
\end{align*}
are well-defined. Moreover, since $\varphi''=-\frac{2 b \varphi'}{\sigma^2}$, $\varphi$ satisfies the following PDE
\begin{align*}
b(x)\varphi'(x) + \frac{1}{2} \sigma^2(x) \varphi''(x) = 0.
\end{align*}
Define $Y_t:=\varphi(X_t)$ and $Y_t^{(n)}:=\varphi(X_t^{(n)})$.
Then by It\^o's formula we have
\begin{align*}
	Y_t
	= \varphi(x_0) + \int_0^t \varphi'(X_s) \sigma(X_s)dW_s
\end{align*}
and
\begin{align*}
	Y_t^{(n)}
	= \varphi(x_0)
	+\int_0^t \left( \varphi'(X_s^{(n)}) b(X_{\eta_n(s)}^{(n)})+\frac{1}{2}\varphi''(X_s^{(n)}) \sigma^2(X_{\eta_n(s)}^{(n)}) \right) ds
	+ \int_0^t \varphi'(X_s^{(n)}) \sigma(X_{\eta_n(s)}^{(n)})dW_s.
\end{align*}

We will make repeated use of the following elementary lemma.
\begin{Lem}\label{PDE_2}
	Suppose that $b \in L^1(\real)$ and Assumption \ref{Ass_1} (ii) hold.
	Let $C_0 = e^{2\ksm^2 \|b\|_{L^1(\real)}}$.
	\begin{itemize} \item[(i)] For any $x \in \real$,
	\begin{align*}
	\frac{1}{C_0}	\leq \varphi'(x)=\exp(f(x)) \leq C_0.
	\end{align*}
	\item[(ii)] For any $x \in \real$, 
	$$|\varphi''(x)| \leq 2\ksm^2 \|b\|_{\infty} \|\varphi'\|_{\infty} \leq 2\|b\|_\infty \ksm^2 C_0.$$
	\item[(iii)] For any $z,w \in Dom(\varphi^{-1})$,
	\begin{align}\label{PDE_4}
	|\varphi^{-1}(z)-\varphi^{-1}(w)|
	\leq C_0 |z-w|.
	\end{align}
	\end{itemize}
\end{Lem}
The proof of Lemma \ref{PDE_2} is trivial and therefore will be omitted.

\subsection{Yamada and Watanabe approximation technique}
To deal with the H\"older continuity of the diffusion coefficient $\sigma$, we use  Yamada and Watanabe approximation technique (see \cite{Yamada} or \cite{Gyongy}).
For each $\delta \in (1,\infty)$ and $\varepsilon \in (0,1)$, we define a continuous function $\psi _{\delta, \varepsilon}: \real \to \real^+$ with $supp\: \psi _{\delta, \varepsilon}  \subset [\varepsilon/\delta, \varepsilon]$ such that
\begin{align*} 
\int_{\varepsilon/\delta}^{\varepsilon} \psi _{\delta, \varepsilon}(z) dz
= 1 \text{ and } 0 \leq \psi _{\delta, \varepsilon}(z) \leq \frac{2}{z \log \delta}, \:\:\:z > 0.
\end{align*}
Since $\int_{\varepsilon/\delta}^{\varepsilon} \frac{2}{z \log \delta} dz=2$, there exists such a function $\psi_{\delta, \varepsilon}$.
We define a function $\phi_{\delta, \varepsilon} \in C^2(\real;\real)$ by
\begin{align*}
\phi_{\delta, \varepsilon}(x)&:=\int_0^{|x|}\int_0^y \psi _{\delta, 
	\varepsilon}(z)dzdy.
\end{align*}
It is easy to verify that $\phi_{\delta, \varepsilon}$ has the following useful properties: 
\begin{align} 
&|x| \leq \varepsilon + \phi_{\delta, \varepsilon}(x), \text{ for any $x \in \real $}, \label{phi3}\\ 
&0 \leq |\phi'_{\delta, \varepsilon}(x)| \leq 1, \text{ for any $x \in \real$} \label{phi2}, \\
\phi''_{\delta, \varepsilon}(\pm|x|)&=\psi_{\delta, \varepsilon}(|x|)
\leq \frac{2}{|x|\log \delta}{\bf 1}_{[\varepsilon/\delta, \varepsilon]}(|x|), 
\text{ for any $x \in \real \setminus\{0\}$}. \label{phi4}
\end{align}

From \eqref{PDE_4} and \eqref{phi3}, for any $t \in [0,T]$, we have
\begin{align}\label{esti_X1}
	|X_t-X_t^{(n)}|
	\leq C_0 |Y_t-Y_t^{(n)}|
	\leq C_0 \left( \varepsilon + \phi_{\delta,\varepsilon}(Y_t-Y_t^{(n)}) \right).
\end{align}

Using It\^o's formula, we have
\begin{align}\label{esti_X2}
	\phi_{\delta,\varepsilon}(Y_t-Y_t^{(n)})
	=M_t^{n,\delta,\varepsilon}
	+I_t^{(n)}
	+J_t^{(n)},
\end{align}
where
\begin{align*}
M_t^{n,\delta,\varepsilon}
&:=\int_0^t \phi'_{\delta,\varepsilon}(Y_s-Y_s^{(n)}) \left\{ \varphi'(X_s)\sigma(X_s) - \varphi'(X_s^{(n)}) \sigma(X_{\eta_n(s)}^{(n)}) \right\}dW_s,\\
I_t^{(n)}
&:=-\int_0^t \phi'_{\delta,\varepsilon}(Y_s-Y_s^{(n)})
\left\{ \varphi'(X_s^{(n)})b(X_{\eta_n(s)}^{(n)}) +\frac{1}{2} \varphi''(X_s^{(n)}) \sigma^2(X_{\eta_n(s)}^{(n)}) \right\} ds,\\
J_t^{(n)}
&:=\frac{1}{2}\int_0^t \phi''_{\delta,\varepsilon}(Y_s-Y_s^{(n)})
\left| \varphi'(X_s) \sigma(X_s) - \varphi'(X_s^{(n)}) \sigma(X_{\eta_n(s)}^{(n)})  \right|^2 ds.
\end{align*}
In the following, we will estimate $M_t^{n,\delta,\varepsilon}, I_t^{(n)}$ and $J_t^{(n)}$ under various assumptions on $b$ and $\sigma$.

\subsection{Proof of Theorem \ref{Main_1}}

We first consider $I_t^{(n)}$.
Since $\varphi'' = - \frac{2b\varphi'}{\sigma^2}$, 
\begin{align}
	|I_t^{(n)}|
	&\leq \int_0^T \left|\phi'_{\delta,\varepsilon}(Y_s-Y_s^{(n)}) \varphi'(X_s^{(n)}) \right|
	\left|b(X_{\eta_n(s)}^{(n)}) - \frac{b((X_s^{(n)})) \sigma^2(X_{\eta_n(s)}^{(n)})}{\sigma^2(X_s^{(n)})} \right| ds. \notag
\end{align}
Thanks to Lemma \ref{PDE_2} and estimate \eqref{phi2}, we have 
\begin{align}
	|I_t^{(n)}|
	&\leq \ksm^2 C_0 \int_0^T
	\left|b(X_{\eta_n(s)}^{(n)}) \sigma^2(X_s^{(n)}) - b((X_s^{(n)})) \sigma^2(X_{\eta_n(s)}^{(n)}) \right| ds \notag\\
	&\leq \ksm^2 C_0 
	\int_0^T \left\{
	\ksm^2 \left|b(X_s^{(n)}) - b(X_{\eta_n(s)}^{(n)}) \right|
	+\|b\|_{\infty} \left|  \sigma^2(X_s^{(n)}) - \sigma^2(X_{\eta_n(s)}^{(n)}) \right|
	\right\}ds. \notag
\end{align}
It follows from Assumption \ref{Ass_1} that 
\begin{align}\label{esti_I_1}
		|I_t^{(n)}|
&\leq \ksm^4 C_0
	\int_0^T \left\{
	\left|b_A(X_s^{(n)}) - b_A(X_{\eta_n(s)}^{(n)}) \right|
	+\|b_H\|_\beta \left|X_s^{(n)} - X_{\eta_n(s)}^{(n)} \right|^{\beta}
	\right\}ds \notag\\
	&+2\ksm^3 \|b\|_\infty C_0
	\int_0^T 
	\left|  X_s^{(n)} - X_{\eta_n(s)}^{(n)} \right|^{1/2+\alpha} ds.
\end{align}

Now we estimate $J_t^{(n)}$.
From \eqref{phi4}, we have
\begin{align*}
J_t^{n}
&\leq \int_0^T \frac{\1_{[\varepsilon/\delta,\varepsilon]}(|Y_s-Y_s^{(n)}|)}{|Y_s-Y_s^{(n)}| \log \delta}
\left| \varphi'(X_s) \sigma(X_s) - \varphi'(X_s^{(n)}) \sigma(X_{\eta_n(s)}^{(n)})  \right|^2 ds\\
&\leq 3(J_T^{1,n}+J_T^{2,n}+J_T^{3,n}),
\end{align*}
where
\begin{align*}
J_t^{1,n}
&:= \int_0^t \frac{\1_{[\varepsilon/\delta,\varepsilon]}(|Y_s-Y_s^{(n)}|)}{|Y_s-Y_s^{(n)}| \log \delta}
|\sigma(X_s)|^2 \left| \varphi'(X_s) - \varphi'(X_s^{(n)}) \right|^2 ds,\\
J_t^{2,n}
&:=\int_0^t \frac{\1_{[\varepsilon/\delta,\varepsilon]}(|Y_s-Y_s^{(n)}|)}{|Y_s-Y_s^{(n)}| \log \delta}
|\varphi'(X_s^{(n)})|^2 \left| \sigma(X_s) - \sigma(X_s^{(n)})  \right|^2 ds, \\
J_t^{3,n}
&:=\int_0^t \frac{\1_{[\varepsilon/\delta,\varepsilon]}(|Y_s-Y_s^{(n)}|)}{|Y_s-Y_s^{(n)}| \log \delta}
|\varphi'(X_s^{(n)})|^2 \left| \sigma(X_s^{(n)}) - \sigma(X_{\eta_n(s)}^{(n)})  \right|^2 ds.
\end{align*}

From Lemma \ref{PDE_2} (ii), $\varphi'$ is Lipschitz continuous with Lipschitz constant $\|\varphi''\|_{\infty}$.
Hence, we have
\begin{align}\label{esti_J1}
	J_T^{1,n}
	&\leq \frac{\ksm^2 \|\varphi''\|_{\infty}^2}{\log \delta} \int_0^T \frac{\1_{[\varepsilon/\delta,\varepsilon]}(|Y_s-Y_s^{(n)}|)}{|Y_s-Y_s^{(n)}|}
	\left| X_s - X_s^{(n)} \right|^2 ds \notag\\
	&\leq \frac{\ksm^2 \|\varphi''\|_{\infty}^2 C_0^2}{\log \delta}
	\int_0^T\1_{[\varepsilon/\delta,\varepsilon]}(|Y_s-Y_s^{(n)}|) \left| Y_s - Y_s^{(n)} \right| ds \notag\\
	&\leq \frac{4\ksm^6 C_0^4 \|b\|_\infty^2 T \varepsilon}{\log \delta}
\end{align}
and since $\sigma$ is $(\frac 12 +\alpha)$-H\"older continuous, we have
\begin{align}\label{esti_J2}
	J_T^{2,n}
	&\leq \frac{\ksm^2 C_0^2 }{\log \delta} \int_0^T \frac{\1_{[\varepsilon/\delta,\varepsilon]}(|Y_s-Y_s^{(n)}|)}{|Y_s-Y_s^{(n)}|}
	\left| X_s - X_s^{(n)}  \right|^{1+2\alpha} ds \notag\\
	&\leq \frac{\ksm^2 C_0^{3+2\alpha}}{\log \delta} \int_0^T \1_{[\varepsilon/\delta,\varepsilon]}(|Y_s-Y_s^{(n)}|)
	\left| Y_s - Y_s^{(n)}  \right|^{2\alpha} ds \notag\\
	&\leq \frac{\ksm^2 C_0^{3+2\alpha} T \varepsilon^{2\alpha}}{\log \delta},
\end{align}
and
\begin{align}\label{esti_J3}
J_T^{3,n}
&\leq \frac{\ksm^2 C_0^2  \delta }{\varepsilon \log \delta}
 \int_0^T \left|X_s^{(n)} - X_{\eta_n(s)}^{(n)} \right|^{1+2\alpha} ds.
\end{align}
Therefore, from \eqref{esti_X1}, \eqref{esti_X2}, \eqref{esti_I_1}, \eqref{esti_J1}, \eqref{esti_J2} and \eqref{esti_J3}, for any stopping time $\tau \in \mathcal{T}$,
\begin{align}\label{esti_X3}
	&|X_{\tau}-X_\tau^{(n)}| \leq C_0\varepsilon + C_0 M_{\tau}^{n,\delta,\varepsilon} \notag \\
	&+ \ksm^4 C_0^2
	\int_0^T \left\{
	\left|b_A(X_s^{(n)}) - b_A(X_{\eta_n(s)}^{(n)}) \right|
	+\|b_H\|_\beta \left|X_s^{(n)} - X_{\eta_n(s)}^{(n)} \right|^{\beta}
	\right\}ds \notag\\
	&+2\ksm^3 \|b\|_\infty C_0^2
	\int_0^T 
	\left|  X_s^{(n)} - X_{\eta_n(s)}^{(n)} \right|^{1/2+\alpha} ds \notag\\
	&+ \frac{12\ksm^6 C_0^5 \|b\|_\infty^2T \varepsilon}{\log \delta} 
	+ \frac{3\ksm^2 C_0^{4+2\alpha} T \varepsilon^{2\alpha}}{\log \delta} 
	+ \frac{3\ksm^2 C_0^3  \delta }{\varepsilon \log \delta}
 \int_0^T \left|X_s^{(n)} - X_{\eta_n(s)}^{(n)} \right|^{1+2\alpha} ds.
	\end{align}
Note that since $\phi'$, $\varphi'$ and $\sigma$ are bounded, $(M_t^{n,\delta,\varepsilon})_{0 \leq t \leq T}$ is martingale, so expectation of $M_{\tau}^{n,\delta,\varepsilon}$ equals to zero.

If $\alpha \in (0,1/2]$, then by choosing $\varepsilon=n^{-1/2}$ and $\delta =2$, the estimate \eqref{esti_X3} becomes
\begin{align}\label{esti_X3_1}
	&|X_{\tau}-X_\tau^{(n)}| \leq \frac{C_0}{n^{1/2}} + C_0 M_{\tau}^{n,2,n^{-1/2}} \notag \\
	&+ \ksm^4 C_0^2
	\int_0^T \left\{
	\left|b_A(X_s^{(n)}) - b_A(X_{\eta_n(s)}^{(n)}) \right|
	+\|b_H\|_\beta \left|X_s^{(n)} - X_{\eta_n(s)}^{(n)} \right|^{\beta}
	\right\}ds \notag\\
	&+2\ksm^3 \|b\|_\infty C_0^2
	\int_0^T 
	\left|  X_s^{(n)} - X_{\eta_n(s)}^{(n)} \right|^{1/2+\alpha} ds \notag\\
	&+ \frac{12\ksm^6 C_0^5 \|b\|_\infty^2 T }{n^{1/2}\log 2} 
	+ \frac{3\ksm^2 C_0^{4+2\alpha} T }{n^{\alpha}\log 2} 
	+ \frac{6\ksm^2 C_0^3  n^{1/2} }{ \log 2}
 \int_0^T \left|X_s^{(n)} - X_{\eta_n(s)}^{(n)} \right|^{1+2\alpha} ds.
\end{align}
By taking an expectation in \eqref{esti_X3_1} it follows from Lemma \ref{Lem} and Lemma \ref{Lem_1} with $q=1$ that
\begin{align}\label{esti_X5}
\sup_{\tau \in \mathcal{T}} \e[|X_{\tau}-X_\tau^{(n)}|] 
&\leq \frac{C_1}{n^{\frac{\beta}{2} \wedge \alpha}},
\end{align}
where
\begin{align*}
	C_1 = \ksm^6 C_0^5 \Big(  1 + \frac{3T(1+4\|b\|_\infty^2)}{\log 2} + C(1 + T\|b_H\|_\beta + 2\|b\|_\infty T + \frac{6T}{\log 2})\Big),
\end{align*}
where $C$ is a constant depending on $\alpha, \beta, T, \|b_A\|_{\mathcal{A}}, K_\sigma, \|b\|_\infty$ and $x_0$. Note that $C$ doesn't depend on $\|b\|_{L_1(\real)}$. 
This concludes \eqref{est_l1} for $\alpha \in (0,1/2]$.

If $\alpha =0$, then by choosing $\varepsilon=(\log n)^{-1}$ and $\delta = n^{1/3}$, the estimate \eqref{esti_X3} becomes
\begin{align}\label{esti_X3_2}
&|X_{\tau}-X_\tau^{(n)}| \leq \frac{C_0}{\log n}  + C_0 M_{\tau}^{n, n^{1/3},(\log n)^{-1}} \notag \\
	&+ \ksm^4 C_0^2
	\int_0^T \left\{
	\left|b_A(X_s^{(n)}) - b_A(X_{\eta_n(s)}^{(n)}) \right|
	+\|b_H\|_\beta \left|X_s^{(n)} - X_{\eta_n(s)}^{(n)} \right|^{\beta}
	\right\}ds \notag\\
	&+2\ksm^3 \|b\|_\infty C_0^2
	\int_0^T 
	\left|  X_s^{(n)} - X_{\eta_n(s)}^{(n)} \right|^{1/2} ds \notag\\
	&+ \frac{36\ksm^6 C_0^5 \|b\|_\infty^2 T }{(\log n)^2} 
	+ \frac{9\ksm^2 C_0^{4} T }{\log n} 
	+ 9\ksm^2 C_0^3  n^{1/3}  \int_0^T \left|X_s^{(n)} - X_{\eta_n(s)}^{(n)} \right| ds.
	\end{align}
By taking an expectation in \eqref{esti_X3_2}, it follows from Lemma \ref{Lem} and \ref{Lem_1} with $q=1$ that
	\begin{align}\label{esti_X6}
	\sup_{\tau \in \mathcal{T}} \e[|X_{\tau}-X_\tau^{(n)}|]
	\leq \frac{C_2}{\log n},
\end{align}
where
\begin{align*}
	C_2:=&\ksm^6 C_0^5 \Big( 1 + 9T + C \big( 1 + T \|b_H\|_\beta + 2T\|b\|_\infty + 36T\|b\|_\infty^2 + 9T\big) \Big),
\end{align*}
where $C$ is defined as in \eqref{esti_X5}. This concludes \eqref{est_l1} for $\alpha =0$.
Finally, the estimate \eqref{est_l1b} follows directly from \eqref{est_l1} and Lemma 3.2 in \cite{GKb}.
\qed

\subsection{Proof of Theorem \ref{Main_4}}
The main idea of the proof is to approximate $b$ by a sequence of functions $(b_m)_{m \in \mathbb{N}} \subset  L^1(\real)$ and to apply Theorem \ref{Main_1} for solution of the SDE with drift coefficient $b_m$. 

For $m \in \n$, we choose a smooth function $g_m \in C^1(\real)$ with support $[-(m+2),m+2]$ such that $g_m=1$ on $[-m,m]$ and $0 \leq g_m(x) \leq 1$ for all $x \in \real$, and $\|g_m'\|_\infty \leq 1$.\\ 
We define $b_m:=b  g_m$. It is easy to verify that 
\begin{itemize} 
\item $\|b_m\|_\infty \leq \|b\|_\infty$;
\item  $b_m \in L^{1}(\real)$ and $\|b_m\|_{L^1(\real)} \leq (2m+2)\|b\|_\infty$; 
\item  $b_A  g_m \in \mathcal{A}$ and  $\|b_A g_m\|_{\mathcal{A}} \leq 3\|b_A\|_{\mathcal{A}}$;
\item $b_H  g_m \in H^{\beta}$ and $\|b_Hg_m\|_\beta \leq 2\|b_H\|_\beta$  for all $m$.
\end{itemize}
Let $\overline{X}^m$ and $\overline{X}^{m,n}$ be a unique solution of SDE \eqref{SDE_1} with drift $b_m$ and its Euler-Maruyama approximation, respectively.
Then it holds that
\begin{align}\label{esti_X7}
	\sup_{\tau \in \mathcal{T}} \e[|X_{\tau}-X_\tau^{(n)}|]
	&\leq \sup_{\tau \in \mathcal{T}}
	\left\{
	\e[|X_{\tau}-\overline{X}_\tau^{m}|]
	+\e[|\overline{X}_\tau^{m,n}-X_\tau^{(n)}|]
	+\e[|\overline{X}_\tau^{m}-\overline{X}_\tau^{m,n}|]
	\right\}.
\end{align}
From \eqref{esti_X5} and \eqref{esti_X6}, it holds that
\begin{align}\label{esti_X8}
\sup_{\tau \in \mathcal{T}} \e[|\overline{X}_\tau^{m}-\overline{X}_\tau^{m,n}|]
&\leq
\left\{ \begin{array}{ll}
\displaystyle \frac{C_3 e^{20\ksm^2 \|b\|_\infty m}}{\log n} &\text{ if } \alpha = 0,  \\
\displaystyle \frac{C_3 e^{20\ksm^2 \|b\|_\infty m}}{n^{\frac{\beta}{2} \wedge \alpha}} &\text{ if } \alpha \in (0, 1/2],
\end{array}\right.
\end{align}
where $C_3\equiv C_3(x_0, \ksm, \|b\|_\infty, T, \alpha, \beta)$ is a finite constant which  depends neither on $n$ nor on $m$.
On the other hand, for any stopping time $\tau \in \mathcal{T}$, it holds that
\begin{align*}
	\{ X_{\tau} \neq \overline{X}_{\tau}^{m} \}
	\subset \{ \sup_{0 \leq t \leq \tau} |X_{t}| \geq m \}
	\subset \left\{ \sup_{0 \leq t \leq \tau} \left|\int_0^t \sigma(X_s)dW_s \right| \geq m - |x_0| - \|b\|_{\infty} T \right\}.
\end{align*}
Since $\langle \int_0^{\cdot} \sigma(X_s) dW_s \rangle_{\tau} \leq \ksm^2 T$ almost surely, from Proposition 6.8 in \cite{Shigekawa}, we have
\begin{align*} 
	\p( X_{\tau} \neq \overline{X}_{\tau}^{m})
	&\leq 2 \exp\left(- \frac{(m-|x_0|-\|b\|_{\infty} T)^2}{2 \ksm^2T} \right) \notag\\
	&\leq 2 \exp\left( \frac{(|x_0|+\|b\|_{\infty} T)^2}{2 \ksm^2T} \right) \exp\left(- \frac{m^2}{4 \ksm^2T} \right).
\end{align*}
Since $\mathbb{E}[\sup_{0 \leq t \leq T} |X_t|^2] \vee \mathbb{E}[\sup_{0 \leq t \leq T} |\overline{X}^m_t|^2] \leq C_4 := 3(x_0^2 + T^2 \|b\|_\infty^2 + 12 \ksm^2T)$, we have 
\begin{align} \label{esti_X10}
\Big(\mathbb{E}[|X_\tau - \overline{X}^m_\tau|] \Big)^2 &=\Big(\mathbb{E}[|X_\tau - \overline{X}^m_\tau| \1_{\{ X_{\tau} \neq \overline{X}_{\tau}^{m}\}} ] \Big)^2 \notag \\
 &\leq \mathbb{E}[|X_\tau - \overline{X}^m_\tau|^2]  \p( X_{\tau} \neq \overline{X}_{\tau}^{m}) \notag\\
& \leq C_5^2 \exp\left(- \frac{m^2}{4 \ksm^2T} \right),
\end{align}
where $C_5^2 =  4C_4\exp\left( \frac{(|x_0|+\|b\|_{\infty} T)^2}{2 \ksm^2T} \right)$. 
In the same way, we have
\begin{align} \label{esti_X11}
\Big(\mathbb{E}[|X^{(n)}_\tau - \overline{X}^{m,n}_\tau|] \Big)^2 
& \leq C_5^2 \exp\left(- \frac{m^2}{4 \ksm^2T} \right).
\end{align}

If $\alpha \in (0, \frac 12]$, from \eqref{esti_X7}, \eqref{esti_X8}, \eqref{esti_X10} and \eqref{esti_X11}, we have
$$\sup_{\tau \in \mathcal{T}} \e[|X_{\tau}-X_\tau^{(n)}|] \leq C_3 \frac{e^{20\ksm^2 \|b\|_\infty m}}{n^{\frac{\beta}{2} \wedge \alpha}} + 2C_5 \exp\Big(-\frac{m^2}{8\ksm^2 T}\Big).$$
Choose $m^2 = 8\big(\frac{\beta}{2} \wedge \alpha\big)\ksm^2T \log n$, we obtain 
$$\sup_{\tau \in \mathcal{T}} \e[|X_{\tau}-X_\tau^{(n)}|]  \leq  \frac{C_3 e^{C_6 \sqrt{\log n}}}{n^{\frac{\beta}{2} \wedge \alpha}} + \frac{2C_5}{n^{\frac{\beta}{2} \wedge \alpha}},$$
where $C_6 = 40 \ksm^3 \|b\|_\infty \sqrt{2\big(\frac{\beta}{2} \wedge \alpha\big)T}$. This concludes \eqref{est_l4}  for $\alpha\in(0,1/2]$.

If $\alpha = 0$, from \eqref{esti_X7}, \eqref{esti_X8}, \eqref{esti_X10} and \eqref{esti_X11}, we have
$$\sup_{\tau \in \mathcal{T}} \e[|X_{\tau}-X_\tau^{(n)}|] \leq \frac{C_3 e^{20\ksm^2 \|b\|_\infty  m}}{\log n} + 2C_5 \exp \Big(-\frac{m^2}{8\ksm^2 T}\Big).$$
Choose $m^2 = 8\ksm^2T \log(\log n)$, we obtain 
$$\sup_{\tau \in \mathcal{T}} \e[|X_{\tau}-X_\tau^{(n)}|] \leq \frac{C_3 e^{C_7\sqrt{\log (\log n)}}}{\log n} +  \frac{2C_5}{\log n},$$
where $C_7 = 40\sqrt{2T}\ksm^3\|b\|_\infty $. This concludes\eqref{est_l4} for $\alpha =0$.

Finally, the estimate \eqref{est_l4b} follows directly from \eqref{est_l4} and Lemma 3.2 in \cite{GKb}.

\qed

\subsection{Proof of Theorem \ref{Main_2}}

Define $V_t:=\sup_{ 0 \leq s \leq t} |X_s-X_s^{(n)}|$.
To estimate the expectation of $V_T$, we need to estimate the expectation of $\sup_{0 \leq s \leq T} |M_s^{n,\delta,\varepsilon}|$.
Using Burkholder-Davis-Gundy's inequality, we have
\begin{align*}
	\e\left[ \sup_{0 \leq s \leq T} |M_s^{n,\delta,\varepsilon}| \right]
	&\leq \sqrt{32} \e\left[ \langle M^{n,\delta,\varepsilon} \rangle_T^{1/2} \right] \\
	&= \sqrt{32} \e\left[ \left( \int_0^T \left|\phi'_{\delta,\varepsilon}(Y_t-Y_t^{(n)})\right|^2 \left| \varphi'(X_s)\sigma(X_s) - \varphi'(X_s^{(n)}) \sigma(X_{\eta_n(s)}^{(n)}) \right|^2 ds \right)^{1/2} \right].
\end{align*}
Using the fact that $\|\phi'\|_\infty \leq 1$, we have 
\begin{align}
	&\e\left[ \sup_{0 \leq s \leq T} |M_s^{n,\delta,\varepsilon}| \right] \notag\\
	&\leq \sqrt{96}\ksm \e\left[
	\left(
		\int_0^T
		\left| 
			\varphi'(X_s)-\varphi'(X_s^{(n)})	
		\right|^2ds
	 \right)^{1/2}
	 \right] \notag
	 + \sqrt{96} C_0 \e\left[
	 \left( 
		 \int_0^T
		 \left| 
			\sigma(X_s)-\sigma(X_s^{(n)})
		 \right|^2ds
	 \right)^{1/2}
	 \right]\notag\\
	 &+\sqrt{96} C_0 \e\left[
	 \left( 
		 \int_0^T
		 \left| 
			\sigma(X_s^{(n)})-\sigma(X_{\eta_n(s)}^{(n)})
		 \right|^2ds
	 \right)^{1/2}
	 \right]. \notag
	 \end{align}
Since $\sigma$ is H\"older continuous and $\varphi'$ is Lipschitz,  $\e\left[ \sup_{0 \leq s \leq T} |M_s^{n,\delta,\varepsilon}| \right]$ is bounded by 
\begin{align}\label{esti_M_2_2}
 &\sqrt{96} \ksm \|\varphi''\|_{\infty} \e\left[
	 \left(
	 \int_0^T
	 \left| 
		 X_s-X_s^{(n)}
	 \right|^2ds
	 \right)^{1/2}
	 \right]\notag
	 +\sqrt{96} C_0 \ksm \e\left[
	 \left( 
	 \int_0^T
	 \left| 
		X_s-X_s^{(n)}
	 \right|^{1+2\alpha}ds
	 \right)^{1/2}
	 \right]\notag\\
	 &+\sqrt{96} C_0 \ksm  \left(
	 \int_0^T
	 \e\big[|  X_s^{(n)}-X_{\eta_n(s)}^{(n)}|^{1+2\alpha}\big]ds
	 \right)^{1/2}\notag\\
	 &\leq \tilde{c}_1
	 \e\left[
		 A_T^{(n)}
	 \right]	
	 + \tilde{c}_2
	 \e\left[
		 B_T^{(n)}
	 \right]
	 + \frac{\tilde{c}_3}{n^{1/4+\alpha/2}},
\end{align}
where $\tilde{c}_1:=\sqrt{96}\ksm \|\varphi''\|_{\infty}$, $\tilde{c}_2:=\sqrt{96}C_0 \ksm$, $\tilde{c}_3:=\sqrt{96CT}C_0 \ksm$,
\begin{align*}
	A_T^{(n)}
	&:= \left(
	\int_0^T
	\left| 
	X_s-X_s^{(n)}
	\right|^2ds
	\right)^{1/2} \text{ and }
	B_T^{(n)}
	:=\left( 
	\int_0^T
	\left| 
	X_s-X_s^{(n)}
	\right|^{1+2\alpha}ds
	\right)^{1/2}.
\end{align*}
Since $|X_s-X_s^{(n)}| \leq V_T^{(n)}$ for any $s \in [0,T]$, by using Young's inequality $xy \leq \frac{x^2}{4 \tilde{c}_1 C_0} + \tilde{c}_1 C_0y^2$, we have
\begin{align*}
	A_T^{(n)}
	&\leq \left(V^{(n)}_T\right)^{1/2} \left(
	\int_0^T
	\left| 
	X_s-X_s^{(n)}
	\right|ds
	\right)^{1/2}
	\leq \frac{V_T^{(n)}}{4 \tilde{c}_1 C_0} + \tilde{c_1} C_0 \int_0^T |X_s-X_s^{(n)}| ds.
\end{align*}
Hence it holds that
\begin{align}\label{esti_A_2}
	\e[A_T^{(n)}]
	\leq \frac{\e[V_T^{(n)}]}{4 \tilde{c}_1 C_0} + \tilde{c_1} C_0 \int_0^T \e[|X_s-X_s^{(n)}|] ds.
\end{align}
Next we estimate the expectation of $B_T^{(n)}$ and $V_T^{(n)}$.

For $\alpha \in (0,1/2]$, by using Young's inequality $xy \leq \frac{x^2}{4 \tilde{c}_2 C_0} + \tilde{c}_2 C_0 y^2$, we have
\begin{align*}
B_T^{(n)}
&\leq \left(V^{(n)}_T\right)^{1/2} \left(
\int_0^T
\left| 
X_s-X_s^{(n)}
\right|^{2\alpha}ds
\right)^{1/2}
\leq \frac{V_T^{(n)}}{4 \tilde{c}_2 C_0} + \tilde{c}_2C_0 \int_0^T |X_s-X_s^{(n)}|^{2 \alpha} ds.
\end{align*}
Hence it holds from Jensen's inequality that
\begin{align}\label{esti_B_2}
\e[B_T^{(n)}]
\leq \frac{\e[V_T^{(n)}]}{4 \tilde{c}_2 C_0} + \tilde{c}_2 C_0 \int_0^T \e[|X_s-X_s^{(n)}|]^{2 \alpha} ds.
\end{align}
Therefore from \eqref{esti_M_2_2}, \eqref{esti_A_2} and \eqref{esti_B_2}, we obtain
\begin{align}
\e\left[ \sup_{0 \leq s \leq T} |M_{s}^{n,2,n^{-1/2}}| \right] \notag
&\leq \frac{\e[V_T^{(n)}]}{2C_0}
+\tilde{c}_1^2 C_0 \int_0^T \e[|X_s-X_s^{(n)}|] ds\\
&+\tilde{c}_2^2 C_0 \int_0^T \e[|X_s-X_s^{(n)}|]^{2\alpha} ds
+ \frac{\tilde{c}_3}{n^{1/4+\alpha/2}}. \notag 
\end{align}
It follows from \eqref{esti_X5} that 
\begin{align*}
\e\left[ \sup_{0 \leq s \leq T} |M_{s}^{n,2,n^{-1/2}}| \right]
&\leq \frac{\e[V_T^{(n)}]}{2C_0}
+\frac{ (\tilde{c}_1^2 +\tilde{c}_2^2 )T C_0 C_1}{n^{\alpha(\beta \wedge 2\alpha)} }
+ \frac{\tilde{c}_3}{n^{1/4+\alpha/2}}.
\end{align*}
This fact together with \eqref{esti_X3_1} and \eqref{esti_X5} implies that 
\begin{align*}
&\e[V_T^{(n)}]
\leq \frac{ 2 (\tilde{c}_1^2 +\tilde{c}_2^2 )T C_0^2 C_1}{n^{\alpha(\beta \wedge 2\alpha)} }
+ \frac{2 \tilde{c}_3C_0}{n^{1/4+\alpha/2}}
+\frac{2 C_1}{n^{\frac{\beta}{2} \wedge \alpha}}.
\end{align*}
Since $1/4+\alpha/2 \geq 2 \alpha^2$, we have
\begin{align} \label{esti_X5b}
&\e[V_T^{(n)}] \leq \frac{C_8}{n^{\alpha(\beta \wedge 2\alpha)}},
\end{align}
where
$C_8:=2 (\tilde{c}_1^2 +\tilde{c}_2^2 )T C_0^2 C_1	+ 2 \tilde{c}_3C_0	+ 2 C_1.$

For $\alpha = 0$, by Jensen's inequality we have
\begin{align}\label{esti_B_3}
	\e[B_T^{(n)}]
	\leq \left( \int_0^T\e\left[\left|X_s-X_s^{(n)}\right|\right] ds\right)^{1/2}
	\leq \frac{\sqrt{C_2T}}{\sqrt{\log n}}.
\end{align}
Therefore from \eqref{esti_M_2_2}, \eqref{esti_A_2} and \eqref{esti_B_3}, we obtain
\begin{align*}
\e\left[ \sup_{0 \leq s \leq T} \left|M_{s}^{n,n^{1/3},(\log n)^{-1}} \right| \right] \notag
&\leq \frac{\e[V_T]}{4C_0}
+\tilde{c}_1^2 C_0 \int_0^T \e[|X_s-X_s^{(n)}|] ds
+\frac{\tilde{c}_2 \sqrt{C_2T}}{\sqrt{\log n}}
+ \frac{\tilde{c}_3}{n^{1/4+\alpha/2}} \notag \\
&\leq \frac{\e[V_T]}{4C_0}
+\frac{\tilde{c}_1^2 C_0 C_2T}{ \log n}
+\frac{\tilde{c}_2 \sqrt{C_2T}}{\sqrt{\log n}}
+ \frac{\tilde{c}_3}{n^{1/4+\alpha/2}}.
\end{align*}
This fact  together with estimates \eqref{esti_X3_2} and \eqref{esti_X6} implies that 
\begin{align} \label{esti_X6b}
\e[V_T^{(n)}]
\leq \frac{4\tilde{c}_1^2 C_0^2 C_2T}{ 3 \log n}
+\frac{4C_0 \tilde{c}_2 \sqrt{C_2T}}{3 \sqrt{\log n}}
+ \frac{4 C_0 \tilde{c}_3}{3 n^{1/4+\alpha/2}}
+ \frac{4C_2}{3 \log n} 
\leq \frac{C_9}{\sqrt{\log n}},
\end{align}
where
$C_9:=\frac 43 ( \tilde{c}_1^2 C_0^2 C_2T
+C_0 \tilde{c}_2 \sqrt{C_2T}
+  C_0 \tilde{c}_3
+ C_2).$
Hence we finish the proof of Theorem \ref{Main_2}.
\qed

\subsection{Proof of Theorem \ref{Main_4b}}
We denote $b_m, \overline{X}^m$ and $\overline{X}^{m,n}$ as in the proof of Theorem \ref{Main_4}. Then it holds that 
\begin{align}
&	 \e[\sup_{0 \leq t \leq T} |X_{t}-X_t^{(n)}|] \notag\\ 
	&\leq 	\left\{
	\e[\sup_{0 \leq t \leq T} |X_{t}-\overline{X}_t^{m}|]
	+\e[\sup_{0 \leq t \leq T} |\overline{X}_t^{m,n}-X_t^{(n)}|]
	+\e[\sup_{0 \leq t \leq T} |\overline{X}_t^{m}-\overline{X}_t^{m,n}|]
	\right\}. \label{esti_X7b}
\end{align}
From \eqref{esti_X5b} and \eqref{esti_X6b}, it holds that
\begin{align}\label{esti_X8b}
\e[\sup_{0 \leq t \leq T} |\overline{X}_t^{m}-\overline{X}_t^{m,n}|]
&\leq
\left\{ \begin{array}{ll}
\displaystyle \frac{C_{10} e^{72\ksm^2 \|b\|_\infty m}}{\sqrt{\log n}} &\text{ if } \alpha = 0,  \\
\displaystyle \frac{C_{10} e^{72\ksm^2 \|b\|_\infty m}}{n^{\alpha (\beta \wedge 2 \alpha)}} &\text{ if } \alpha \in (0, 1/2],
\end{array}\right.
\end{align}
where $C_{10}\equiv C_{10}(x_0, \ksm, \|b\|_\infty, T, \alpha, \beta)$ is a finite constant which  depends neither on $n$ nor on $m$.

Since 
$\{ \sup_{0 \leq t \leq T} |X_{t}-\overline{X}_t^{m}|  \neq 0 \} \subset  \{ \sup_{0 \leq t \leq T} |X_{t}| \geq m \},$
it follows from a similar argument as in the proof of Theorem \ref{Main_2} that 
\begin{align} \label{esti_X10b}
\Big(\mathbb{E}[\sup_{0 \leq t \leq T} |X_t - \overline{X}^m_t|] \Big)^2 
 \leq C_5^2 \exp\left(- \frac{m^2}{4 \ksm^2T} \right),
\end{align}
where $C_5^2 =  4C_4\exp\left( \frac{(|x_0|+\|b\|_{\infty} T)^2}{2 \ksm^2T} \right)$. 
In the same way, we have
\begin{align} \label{esti_X11b}
\Big(\mathbb{E}[\sup_{0 \leq t \leq T}|X^{(n)}_t - \overline{X}^{m,n}_t|] \Big)^2 
& \leq C_5^2 \exp\left(- \frac{m^2}{4 \ksm^2T} \right).
\end{align}

If $\alpha \in (0, \frac 12]$, from \eqref{esti_X7b}, \eqref{esti_X8b}, \eqref{esti_X10b} and \eqref{esti_X11b}, we have
$$\e[\sup_{0 \leq t \leq T} |X_t-X_t^{(n)}|] \leq C_{10} \frac{e^{72\ksm^2 \|b\|_\infty m}}{n^{\alpha(\beta \wedge 2\alpha)}} + 2C_5 \exp\Big(-\frac{m^2}{8\ksm^2 T}\Big).$$
Choose $m^2 = 8\alpha (\beta \wedge 2\alpha)\ksm^2T \log n$, we obtain 
$$\e[\sup_{0 \leq t \leq T} |X_t-X_t^{(n)}|] \leq  \frac{C_{10} e^{C_{11} \sqrt{\log n}}}{n^{\frac{\beta}{2} \wedge \alpha}} + \frac{2C_5}{n^{\frac{\beta}{2} \wedge \alpha}},$$
where $C_{11}  = 144 \ksm^3 \|b\|_\infty \sqrt{2T\alpha(\beta \wedge 2\alpha)}$. This concludes the statement for $\alpha\in(0,1/2]$.

If $\alpha = 0$, from \eqref{esti_X7b}, \eqref{esti_X8b}, \eqref{esti_X10b} and \eqref{esti_X11b}, we have
$$\e[\sup_{0 \leq t \leq T} |X_t-X_t^{(n)}|] \leq \frac{C_{10} e^{72\ksm^2 \|b\|_\infty  m}}{\sqrt{\log n}} + 2C_5 \exp \Big(-\frac{m^2}{8\ksm^2 T}\Big).$$
Choose $m^2 = 8\ksm^2T \log(\log n)$, we obtain 
$$\e[\sup_{0 \leq t \leq T} |X_t-X_t^{(n)}|] \leq \frac{C_3 e^{C_{12}\sqrt{\log (\log n)}}}{\sqrt{\log n}} +  \frac{2C_5}{\sqrt{\log n}},$$
where $C_{12} = 144\sqrt{2T}\ksm^3\|b\|_\infty $. This concludes the statement for $\alpha =0$.

\subsection{Proof of Theorem \ref{Main_3}}

To prove Theorem \ref{Main_3}, we need  the following Gronwall type inequality.	
\begin{Lem}[\cite{Gyongy} Lemma 3.2.] \label{Lem2_1}
	Let $(Z_t)_{t \geq 0}$ be a nonnegative continuous stochastic process and set $V_t:= \sup_{s\leq t}Z_s$.
	Assume that for some $r>0$, $q\geq 1$, $\rho \in [1,q]$ and some constants $\overline{C}_0$ and $\xi \geq 0$, 
	\begin{align*}
	\e[V_t^r] \leq \overline{C}_0 \e\left[ \left( \int_0^t V_s ds \right)^r\right] + \overline{C}_0 \e\left[ \left( \int_0^t Z^{\rho}_s ds \right)^{r/q} \right] +\xi < \infty
	\end{align*}
	for all $t\geq 0$. Then for each $T \geq 0$ the following statements hold.\\
	(i) If $\rho=q$ then there exists a constant $\overline{C}_1$ depending on $\overline{C}_0,T,q$ and $r$ such that
	\begin{align*}
	\e[V_T^r] \leq \overline{C}_1 \xi.
	\end{align*}
	(ii) If $r\geq q$ or $q+1-\rho < r < q $ hold, then there exists constant $\overline{C}_2$ depending on $\overline{C}_0, T, \rho, q $ and $r$, such that
	\begin{align*}
	\e[V_T^r] \leq \overline{C}_2 \xi + \overline{C}_2 \int_0^T\e[Z_s]ds.
	\end{align*}
\end{Lem}

\begin{proof}[Proof of Theorem \ref{Main_3}]
Throughout this proof, the letter $K$ denotes some positive constant whose value can change from line to line.
The constant $K$ may depend on $K_\sigma, \|b_A\|_{\mathcal{A}},$ $\|b_H\|_\beta, \|b\|_{\infty}, \|b\|_{L^1(\real)}, T, x_0, \alpha$ and $\beta$ but it does not depend on $n$.
We will use again the estimates \eqref{esti_X1} and \eqref{esti_X2}.
Let us first consider  the expectation of $\sup_{0 \leq s \leq T} |M_s^{n,\delta,\varepsilon}|^{p}$.
Using Burkholder-Davis-Gundy's inequality, for any $t \in [0,T]$, $\delta \in (1,\infty)$ and $\varepsilon \in (0,1)$, we have
\begin{align*} 
	\e\left[ \sup_{0 \leq s \leq t} |M_s^{n,\delta,\varepsilon}|^p \right]
	&\leq K \e\left[ \left( \int_0^t \left| \varphi'(X_s)\sigma(X_s) - \varphi'(X_s^{(n)}) \sigma(X_{\eta_n(s)}^{(n)}) \right|^2 ds \right)^{p/2} \right].
\end{align*}
Note that for any bounded Lipschitz continuous function $g$ with Lipschitz constant $L>0$, it holds that
\begin{align}\label{propaty_holder}
	\sup_{x,y \in \real, x \neq y} \frac{|g(x)-g(y)|}{|x-y|^{\gamma}} \leq (2 \|g\|_{\infty})^{1-\gamma} L^{\gamma}, \text{ for any } \gamma \in (0,1].
\end{align}
Since $\sigma$ is bounded and $1/2+\alpha$-H\"older continuous, $\phi'$ is bounded by $1$ and $\varphi', \varphi''$ are bounded, by using \eqref{propaty_holder} for $g=\varphi'$ with $\gamma=1/2+\alpha$, we have
\begin{align}\label{esti_M_3_2}
\e\left[ \sup_{0 \leq s \leq t} |M_s^{n,\delta,\varepsilon}|^p \right]
&\leq K
\e\left[
\left(\int_0^t
\left| \varphi'(X_s) - \varphi'(X_s^{(n)})
\right|^2 ds \right)^{p/2} \right] \notag \\
&+ K\e\left[ \left( \int_0^t
\left| \sigma(X_s) - \sigma(X_s^{(n)})
\right|^2 ds \right)^{p/2}  \right] \notag \\
&+ K\e\left[
\left( \int_0^t
\left| \sigma(X_s^{(n)}) - \sigma(X_{\eta_n(s)}^{(n)})
\right|^2 ds \right)^{p/2} \right] \notag\\
&\leq K \e\left[ \left( \int_0^t
\left|
X_s - X_s^{(n)}
\right|^{1+2\alpha} ds \right)^{p/2}  \right] + K\e\left[
\left( \int_0^T
\left| X_s^{(n)} - X_{\eta_n(s)}^{(n)}
\right|^{1+2\alpha} ds \right)^{p/2} \right] \notag\\
&\leq
K \e\left[ \left(\int_0^t \left|X_s - X_s^{(n)}\right|^{1+2\alpha} ds \right)^{p/2} \right]
+ \frac{K}{n^{p/4+p\alpha/2}}.
\end{align}

For $\alpha \in (0,1/2]$, it follows from \eqref{esti_X3_1} that for any $t \in [0,T]$, we have
\begin{align}\label{esti_X_sup_p_1}
\left| V_t^{(n)}\right|^{p} \leq &
K \Bigg\{
\frac{1}{n^{p/2}} +  \sup_{0 \leq s \leq t} \left|M_{s}^{n,2,n^{-1/2}} \right|^p + \int_0^T \left\{
\left|b_A(X_s^{(n)}) - b_A(X_{\eta_n(s)}^{(n)}) \right|^p
+K\left|X_s^{(n)} - X_{\eta_n(s)}^{(n)} \right|^{p \beta}
\right\}ds \notag \\ 
 &+ \int_0^T 
\left|  X_s^{(n)} - X_{\eta_n(s)}^{(n)} \right|^{p/2+p\alpha} ds 
+ \frac{1}{n^{p/2}}
+ \frac{1}{n^{p\alpha}} 
+  n^{p/2}
\int_0^T \left|X_s^{(n)} - X_{\eta_n(s)}^{(n)} \right|^{p+2p\alpha} ds
\Bigg\}.
\end{align}
By taking the expectation on \eqref{esti_X_sup_p_1}, from Lemma \ref{Lem} and \ref{Lem_1}, we have
\begin{align}\label{esti_X_sup_p_2}
	\e\left[ \left| V_t^{(n)}\right|^{p} \right]
	&\leq K \e\left[ \sup_{0 \leq s \leq t} \left|M_{s}^{n,2,n^{-1/2}} \right|^p\right]
	+ \frac{K}{n^{\frac12 \wedge \frac{p\beta}{2}\wedge p \alpha}}.
\end{align}
Since $\alpha \leq 1/4+\alpha/2$, from \eqref{esti_M_3_2} and \eqref{esti_X_sup_p_2}, we obtain
\begin{align}\label{esti_X_sup_p_3}
\e\left[ \left| V_s^{(n)}\right|^{p} \right]
&\leq K
\e\left[ \left(\int_0^t \left|X_s - X_s^{(n)}\right|^{1+2\alpha} ds\right)^{p/2} \right]
+ \frac{K}{n^{\frac12 \wedge \frac{p\beta}{2}\wedge p \alpha}}.
\end{align}
If $\alpha=1/2$,  from Jensen's inequality, we have
\begin{align*}
\e\left[ \left| V_t^{(n)}\right|^{p} \right]
\leq K
 \int_0^t \e\left[ \left|X_s - X_s^{(n)}\right|^{p}  \right]ds
+ \frac{K}{n^{\frac12 \wedge \frac{p\beta}{2}}} 
\leq K \int_0^t \e\left[ \left| V_s^{(n)}\right|^{p}  \right]ds
+ \frac{K}{n^{\frac12 \wedge \frac{p\beta}{2}}}.
\end{align*}
Using Gronwall's inequality, we have
\begin{align} \label{finalK1}
	\e\left[ \left| V_s^{(n)}\right|^{p} \right]
	\leq \frac{K}{n^{\frac12 \wedge \frac{p\beta}{2}}}.
\end{align}
If $\alpha \in (0,1/2)$, then from \eqref{esti_X_sup_p_3}, for any $t \in [0,T]$ we have
\begin{align*}
\e\left[ \left| V_s^{(n)}\right|^{p} \right]
&\leq K \e\left[ \left( \int_0^t V_s^{(n)} ds \right)^p\right]
+K \e\left[ \left(\int_0^t \left|X_s - X_s^{(n)}\right|^{1+2\alpha} ds \right)^{p/2} \right] + \frac{K}{n^{\frac12 \wedge \frac{p\beta}{2}\wedge p \alpha}}.
\end{align*}
Using Lemma \ref{Lem2_1}-(ii) with $r=p$, $q=2$, $\rho=1+2\alpha$, $\xi=Kn^{-(\frac12 \wedge \frac{p\beta}{2}\wedge p \alpha)}$ and Theorem \ref{Main_1}, we have
\begin{align} \label{finalK2}
	\e\left[ \left| V_t^{(n)}\right|^{p} \right]
	\leq \frac{K}{n^{\frac12 \wedge \frac{p\beta}{2}\wedge p \alpha}}
	+K \int_0^t \left[\left| X_s-X_s^{(n)} \right|\right]ds
	\leq \frac{K}{n^{\frac12 \wedge \frac{p\beta}{2}\wedge p \alpha}}
	+\frac{K}{n^{\frac{\beta}{2} \wedge \alpha}} 
	\leq \frac{K}{n^{\frac{\beta}{2} \wedge \alpha}}.
\end{align}
This concludes the case of $\alpha \in (0,1/2]$.

For $\alpha=0$, it follows from \eqref{esti_X3_2} that for any $t \in [0,T]$, we have
\begin{align}\label{esti_X_sup_p_6}
	\left| V_t^{(n)} \right|^p
	&\leq
	K\Bigg\{
	\frac{1}{(\log n)^p}
	+  \sup_{0 \leq s \leq t} \left|M_{s}^{n,n^{1/3},(\log n)^{-1}} \right|^p \notag \\
	&+	\int_0^t \left\{
	\left|b_A(X_s^{(n)}) - b_A(X_{\eta_n(s)}^{(n)}) \right|^p
	+K^p\left|X_s^{(n)} - X_{\eta_n(s)}^{(n)} \right|^{p\beta}
	\right\}ds \notag\\
	&+ 	\int_0^t 
	\left|  X_s^{(n)} - X_{\eta_n(s)}^{(n)} \right|^{p/2} ds 
	+ \frac{1}{(\log n)^{2p}} \notag\\
	&+ \frac{1}{(\log n)^p} 
	+n^{p/3} \int_0^T \left|X_s^{(n)} - X_{\eta_n(s)}^{(n)} \right|^p ds	\Bigg\}.
\end{align}
By taking the expectation on \eqref{esti_X_sup_p_6}, from Lemma \ref{Lem} and \ref{Lem_1}, we have
\begin{align}\label{esti_X_sup_p_7}
	\e\left[ \left| V_t^{(n)} \right|^p \right]
	\leq K \e\left[\sup_{0 \leq s \leq t} \left|M_{s}^{n,n^{1/3},(\log n)^{-1}} \right|^p \right]
	+ \frac{K}{(\log n)^{p}}.
\end{align}
From \eqref{esti_M_3_2} and \eqref{esti_X_sup_p_7}, we obtain
\begin{align*}
	\e\left[ \left| V_t^{(n)} \right|^p \right] \leq 
	& K\e\left[ \left(\int_0^t \left|X_s - X_s^{(n)}\right| ds \right)^{p/2} \right]
	+ \frac{K}{(\log n)^{p}}.
\end{align*}
Using Lemma \ref{Lem2_1}-(ii) with $r=p$, $q=2$, $\rho=1$, $\xi=K (\log n)^{-p}$ and Theorem \ref{Main_1}, we have
\begin{align} \label{finalK3}
\e\left[ \left| V_T^{(n)} \right|^p \right]
&\leq 
\frac{K}{(\log n)^{p}}
+K \int_0^T \e\left[\left|X_s - X_s^{(n)}\right| \right]ds \notag \\
&\leq \frac{K}{\log n}.
\end{align}
This concludes the case for $\alpha=0$.
\end{proof}

\begin{Rem}
Note that it is hardly possible to obtain a $L^p$ bound for the error if we remove the condition $b \in L^1(\real)$ by following the method used in the proofs of Theorem \ref{Main_4} and \ref{Main_4b} since
a careful tracking of constants $K$ will show that the  constants $K$ in \eqref{finalK1}, \eqref{finalK2} and \eqref{finalK3}   increase with the order of double exponent with respect to 
 $\|b\|_{L^1(\real)}$. This makes the localization technique for $b$ not applicable.   

\end{Rem}

\section*{Acknowledgements}
The authors thank Prof. Arturo Kohatsu-Higa for his helpful comments.
This research is funded by Vietnam National Foundation for Science and Technology Development (NAFOSTED) under grant number 101.03-2014.14. The second author was supported by  KOKUSAITEKI Research Fund of Ritsumeikan University and JSPS KAKENHI Grant Number 16J00894.

\end{document}